\documentclass[11pt,twoside]{article}
\usepackage{amsmath,amsfonts,amscd,amsthm}

\newtheorem{thm}{Theorem}[section]

\newtheorem{lem}[thm]{Lemma}
\newtheorem{cor}[thm]{Corollary}

\theoremstyle{definition}

\theoremstyle{remark}

\numberwithin{equation}{section}

\providecommand\ufootnote[1]{{\let\thefootnote\relax\footnote[0]{#1}}}

\newcommand{\vespa}{\vspace{1em}}

\newcommand{\cc}{\mathcal C}

\newcommand{\uc}{\mathcal U}

\newcommand{\rb}{\mathbb R}

\newcommand{\ci}{{\mathcal C}^\infty}

\newcommand{\Cal}{\mathcal}
\newcommand{\N}{\mathbb{N}}

\newcommand{\ol}{\overline}

\newcommand{\pa}{\partial}
\newcommand{\opa}{\ol\partial}
\newcommand{\wt}{\widetilde}

\DeclareMathOperator{\supp}{supp} \DeclareMathOperator{\im}{Im}
\DeclareMathOperator{\ke}{Ker}

 \DeclareMathOperator{\Dom}{Dom}

\begin{document}

\title{Global homotopy formulas on $q$-concave $CR$ manifolds for large degrees}

\author{Till BR\"ONNLE, Christine LAURENT-THI\'{E}BAUT  and J\"{u}rgen LEITERER}

\date{Pr\'{e}publication de l'Institut Fourier n$^\circ$~xxx (2000)\\
\vspace{5pt} {\tt
http://www-fourier.ujf-grenoble.fr/prepublications.html }}
\date{}
\maketitle

\ufootnote{\hskip-0.6cm  {\it A.M.S. Classification}~:
32V20.\newline {\it Key words}~: Homotopy formula, Tangential Cauchy
Riemann equation, CR manifold.}

\bibliographystyle{amsplain}

It is well known that homotopy formulas are very useful in complex
anlysis. Such formulas were constructed by means of integral
operators in the 70's by Grauert and Lieb, Henkin, Ramirez, Kerzman
and Stein for the Cauchy-Riemann operator (see the historical notes
in \cite{Ralivre} for more 
details) and later by Airapetjan and
Henkin \cite{AiHe}, Polyakov \cite{Poly}, Barkatou and
Laurent-Thi\'ebaut \cite{BaLa} for the tangential Cauchy-Riemann
operator. In most cases only local formulas were obtained. The
question arises if it is  possible to globalize these formulas?
Gluing together local formulas, it is rather easy to get a global
formula which is not yet a homotopy formula, but ''almost'', up to a
compact perturbation. Then the main work is to eliminate this
compact perturbation. A first step in that direction was done in
\cite{Leglob} and then applied in \cite{LaLeseparation} to get a
global homotopy formula for the Cauchy-Riemann operator in
$q$-concave-$q^*$-convex domains of a complex manifold. More
recently Polyakov \cite{Poly1,Poly2} proved global homotopy formulas
for the forms of small degrees for the tangential Cauchy-Riemann
operator on compact $q$-concave $CR$ manifolds and used them to
study the embedding problem for $CR$ manifolds. But his global
operators are less regular than the local ones. Then in
\cite{LaLeglobal} it was obtained that, in the case of forms of
small degrees, it is possible to eliminate the compact perturbation
without any loss of smoothness.

In the present paper we extend the results of \cite{LaLeglobal} to
the case of the forms of large degree (cf. Theorem \ref{20.4.07a}). The
main tools are the same as in \cite{LaLeglobal}, for example the
functional analytic lemma (see Lemma \ref{10.4.07a}) and an
induction lemma (see Lemma \ref{15.4.07-}), but new difficulties
appear because now the induction does not start with functions but
with forms of positive degree. For that we need the Friedrichs
approximation lemma for first order differential operators, well
known for the $L^2$-topology, in the $\cc^k$-topology. Since it
seems that this approximation result does not exist in the
literature, it was proved by the first author in his Diplomarbeit
\cite{Brodipl}. This proof is given at the end of this paper.

As a corollary we get a Dolbeault isomorphism type result (cf.
Corollary \ref{dolbeault}).

In the case of the Cauchy-Riemann operator on a complex manifold,
the Dolbeault isomorphism says that all the Dolbeault cohomology
groups of bidegree $(p,q)$ for currents, $\ci$-forms or
$\cc^k$-forms are isomorphic to the $q$th cohomology group of the
sheaf of germs of holomorphic $p$-forms. This is a consequence of
the de Rham-Weil isomorphism, of the Dolbeault lemma and of the
holomorphy of the $\opa$-closed $(p,0)$-currents.

Let $M$ be a $q$-concave, $q\ge 1$, CR generic submanifold of real
codimension $k$ of a complex manifold $X$ of complex dimension $n$.
As in the complex case, we have smoothness of $\opa_b$-closed
$(n,0)$-currents and local solvability of the tangential
Cauchy-Riemann equation for forms of bidegree $(n,r)$ with $1\le
r\le q-1$ and $n-k-q+1\le r\le n-k$. For the small degrees, then the
Dolbeault isomorphism for the $\opa_b$-cohomology  follows from the
de Rham-Weil isomorphism. For the large degrees , i.e. bidegree
$(n,r)$ with $n-k-q+1\le r\le n-k$, $\opa_b$-closed currents of
bidegree $(n,r)$ need not to be smooth. Nevertheless the Dolbeault
isomorphism between the $\opa_b$-cohomology for currents and the
$\opa_b$-cohomology for $\cc^\infty$-smooth forms  is proved for
$n-k-q+2\le r\le n-k$ in \cite{LaLeDolbeault} under the additional
hypothesis that the conormal bundle of $M$ in $X$ is trivial and in
\cite{Sareg} without any additional hypothesis. Here we prove that,
if moreover $M$ is compact, the
 $\opa_b$-cohomology group for
$\cc^l$-smooth $(n,r)$-forms, $l\in\N$, and  for $\cc^\infty$-smooth
$(n,r)$-forms are isomorphic in the case of the large degrees,
included $r=n-k-q+1$. In \cite{LaLeDolbeault} the reduction to local
results was based on cohomological algebra arguments, in
\cite{Sareg} on the construction of a regularization formula for
$\opa_b$, here it uses functional analysis.

\section{Global homotopy formula}\label{s01}

In this section, $X$ is a complex manifold and $E$ is a holomorphic
vector bundle on $X$. Further, $M\subseteq X$ is a generic, compact
$CR$ submanifold of class $\Cal C^\infty$ of $X$, $k$ is the real
codimension of $M$ in $X$, and $\Cal O$ is the trivial complex line
bundle on $X$.

If $U\subseteq M$ is an open set, then, for $0\le r\le n-k$, the
following notations are used:
\begin{itemize}
\item[-] $\Cal C^{\infty}_{n,r}(U,E)$ is the Fr\'{e}chet space  of $E$-valued $(n,r)$-forms on $U$
which are of class $\Cal C^\infty$, endowed with the $\Cal
C^\infty$-topology.
\item[-]  $\Cal
Z^{\infty}_{n,r}(U,E)$ is the subspace of all closed forms in $\Cal
C^{\infty}_{n,r}( U,E)$, endowed with the same topology.
\item[-] $\Cal
C^{l+\alpha}_{n,r}(\overline U,E)$, $l\in \N$, $0\le \alpha<1$, is
the Banach space of $l$ times differentiable $E$-valued
$(n,r)$-forms whose derivatives up to order $l$ admit extensions to
$\overline U$ which are H\"older continuous with exponent $\alpha$,
endowed with the $\Cal C^{l+\alpha}$-topology. \item[-] $\Cal
Z^{l+\alpha}_{n,r}(\overline U,E)$ is the subspace of all closed
forms in $\Cal C^{l+\alpha}_{n,r}(\overline U,E)$, endowed with the
same topology.
\item[-] If $r\ge 1$, then $\Cal B^{l+\alpha\rightarrow
l}_{n,r}(M,E)$ is the space of all $f\in \Cal C^{l}_{n,r}(M,E)$ such
that $f=du$ for some $u\in\Cal C^{l+\alpha}_{n,r-1}(M,E)$. Sometimes
we write also  $$\Cal B^{\infty}_{n,r}(M,E):=\Cal
B^{\infty\rightarrow \infty}_{n,r}(M,E):=d\Cal
C^{\infty}_{n,r-1}(M,E).$$
\item[-] $(\Dom d)^0_{n,r}(M,E)$ is the space of all $f\in \Cal
C^0_{n,r}(M,E)$ such that also $df$ is continuous on $M$.
\end{itemize}

\vspace{2mm} If $0<\alpha<1$ and $q$ is an integer with $1\le q\le
n-k$, then we shall say that {\bf condition $H(\alpha,q)$ is
satisfied} if, for each point in $M$, there exist a neighborhood $U$
and linear operators
$$
T_r:\Cal C^0_{n,r}(M,\Cal O)\rightarrow\Cal C^0_{n,r-1}\big(U,\Cal
O\big)\,,\qquad 1\le r\le q~ {\rm and}~ n-k-q+1\le r\le n-k,
$$ with the following two properties:

(i) For all $l\in \N$ and $1\le r\le q$ or $n-k-q+1\le r\le n-k$,
$$
T_r\Big(\Cal C^l_{n,r}(M,\Cal O)\Big)\subseteq \Cal
C^{l+\alpha}_{n,r-1}(\overline U,\Cal O)
$$ and $T_r$ is continuous as an operator between $\Cal
C^l_{n,r}(M,\Cal O)$ and $\Cal C^{l+\alpha}_{n,r}(\overline U,\Cal
O)$.

(ii) If $f\in (\Dom d)^0_{n,r}(M,\Cal O)$, $0\le r\le q-1$, has
compact support in $U$, then, on $U$,
\begin{equation}\label{4.3.07}
f=\begin{cases}T_{1}df&\text{if
}r=0\,,\\dT_rf+T_{r+1}df\qquad&\text{if }1\le r\le q-1~ {\rm or}~
n-k-q+1\le r\le n-k\,.\end{cases}
\end{equation}

If $M$ is $q$-concave in the sense of Henkin \cite{Hecras}, then it
is known since 1981 \cite{Hecras,AiHe} that condition $H(\alpha,q)$
is satisfied for $0<\alpha<1/2$. More recently it was proved in
\cite{BaLa} that then also condition $H(1/2,q)$ is satisfied.

\begin{thm}\label{20.4.07a} Suppose, for some $0<\alpha<1$ and some integer
$q$ with $1\le q\le n-k$,  condition $H(\alpha,q)$ is satisfied.
Then there exist finite dimensional subspaces $\Cal H_r$ of $\Cal
Z^\infty_{n,r}(M,E)$, $1\le r\le q-1$ and $n-k-q+1\le r\le n-k$,
where $\Cal H_0=\Cal Z^\infty_{n,0}(M,E)$, continuous linear
operators
\begin{equation*}
A_r:\Cal C^0_{n,r}(M,E)\rightarrow \Cal C^0_{n,r-1}(M,E)\,,\qquad
1\le r\le q~ {\rm and}~ n-k-q+1\le r\le n-k\,,
\end{equation*}
and continuous linear projections $$ P_r:\Cal
C^0_{n,r}(M,E)\rightarrow \Cal C^0_{n,r}(M,E) \,,\qquad 0\le r\le
q-1~ {\rm and}~ n-k-q+1\le r\le n-k\,,
$$with
\begin{equation}\label{19.4.07'}
\im P_r=\Cal H_r\,,\qquad 0\le r\le q-1~ {\rm and}~ n-k-q+1\le r\le
n-k\,,
\end{equation}and
\begin{equation}\label{20.4.07}
\Cal B_{n,r}^{0\rightarrow 0}(M,E)\subseteq\ke P_r\,,\qquad 1\le
r\le q-1~ {\rm and}~ n-k-q+1\le r\le n-k,
\end{equation} such that:

\vspace{2mm} (i) For all $l\in \N\cup\{\infty\}$ and $1\le r\le q$
or $n-k-q+1\le r\le n-k$,
\begin{equation}\label{19.4.07''}
A_r\Big(\Cal C^l_{n,r}(M,E)\Big)\subseteq \Cal
C^{l+\alpha}_{n,r-1}(M,E)
\end{equation}and  $A_r$ is continuous as operator from $\Cal
C^l_{n,r}(M,E)$ to $\Cal C^{l+\alpha}_{n,r-1}(M,E)$.

\vspace{2mm} (ii) For all $0\le r\le q-1$ or $n-k-q+1\le r\le n-k$
and $f\in (\Dom)^0_{n,r}(M,E)$,
\begin{equation}\label{21.3.07'''}
f-P_rf=\begin{cases}A_{1}df &\text{if }r=0\,,\\
dA_rf+A_{r+1}df\qquad&\text{if }1\le r\le q-1~ {\rm or}~ n-k-q+1\le
r\le n-k\,,.
\end{cases}\end{equation}
\end{thm}

In the case of the small degrees, i.e. for $0\le r\le q-1$, Theorem
\ref{20.4.07a} has been proven in \cite{LaLeglobal}, it remains to
prove the case of the large degrees, i.e. $n-k-q+1\le r\le n-k$. The
main ingredients are the same~: first an almost homotopy formula
obtained by gluing together the local formulas and a functional
analytic lemma, second an inductive process.

Let us recall the almost homotopy formula, which is proven in
\cite{LaLeglobal} for small degrees and whose proof is exactly the
same for large degrees:

\begin{lem}\label{21.03.07} Suppose, for some $0<\alpha<1$ and some integer
$q$ with $1\le q\le n-k$,  condition $H(\alpha,q)$ is satisfied.
Then there exist linear operators
\begin{equation}\label{22.3.07}
T_r:\Cal C^0_{n,r}(M,E)\rightarrow\Cal
C^0_{n,r-1}\big(M,E\big)\,,\qquad 1\le r\le q~ {\rm and}~ n-k-q+1\le
r\le n-k, \end{equation}
 and
\begin{equation}\label{10.4.07e}
K_r:\Cal C^0_{n,r}(M,E)\rightarrow\Cal
C^0_{n,r}\big(M,E\big)\,,\qquad 0\le r\le q-1~ {\rm and}~ n-k-q+1\le
r\le n-k,
\end{equation}
with the following two properties:

(i) For all $l\in \N$,
\begin{equation}\label{11.4.07m}
T_r\Big(\Cal C^l_{n,r}(M,E)\Big)\subseteq \Cal
C^{l+\alpha}_{n,r-1}(M,E)\,,\qquad 1\le r\le q~ {\rm and}~
n-k-q+1\le r\le n-k,
\end{equation}
\begin{equation}\label{11.4.07n}
K_r\Big(\Cal C^l_{n,r}(M,E)\Big)\subseteq \Cal
C^{l+\alpha}_{n,r}(M,E)\,,\qquad 0\le r\le q-1~ {\rm and}~
n-k-q+1\le r\le n-k, \end{equation} the operators $T_r$, $1\le r\le
q$ and $n-k-q+1\le r\le n-k$, are continuous as operators acting
between $\Cal C^l_{n,r}(M,E)$ and $\Cal C^{l+\alpha}_{n,r-1}(M,E)$,
and the operators $K_r$, $0\le r\le q-1$ and $n-k-q+1\le r\le n-k$,
are continuous as operators acting between $\Cal C^l_{n,r}(M,E)$ and
$\Cal C^{l+\alpha}_{n,r}(M,E)$.

(ii) If  $f\in (\Dom d)^0_{n,r}(M,E)$, $0\le r\le q-1$ or
$n-k-q+1\le r\le n-k$,  then on $M$
\begin{equation}\label{21.3.07'}
f+K_rf=\begin{cases}T_{1}df&\text{if
}r=0\,,\\dT_rf+T_{r+1}df\qquad&\text{if }1\le r\le q-1~ {\rm or}~
n-k-q+1\le r\le n-k \,.
\end{cases}
\end{equation}
\end{lem}

and the functional analytic lemma:

\begin{lem}\label{10.4.07a} Let $B_l$, $l\in\N$, be a sequence
of Banach spaces, and let $R:B_0\rightarrow B_0$ be a linear
operator such that, for each $l\in \N$:
\begin{itemize}
\item $B_{l+1}\subseteq B_{l}$ and the
imbedding $B_{l+1}\hookrightarrow B_l$ is continuous,
\item $\bigcap_{\mu\in \N}B_\mu$  is dense in $B_l$,
\item $R(B_l)\subseteq B_l$ and
$R\big\vert_{B_l}$ is compact as an endomorphism of $B_l$.
\end{itemize}

\noindent Then $I+R$ is a Fredholm endomorphism with index zero of
$B_0$ (this is clear, because $R$ is compact as an endomorphism of
$B_0$), and
\begin{equation}\label{14.4.07---}
\ke (I+R)\subseteq \bigcap_{l\in \N}B_l\,.
\end{equation}
\end{lem}

from which one can deduce the next result (cf.\cite{LaLeglobal},
Lemma 5.1)~:

\begin{lem}\label{6.4.07a} Suppose, for some $0<\alpha<1$ and some integer
$q$ with $1\le q\le n-k$,  condition $H(\alpha,q)$ is satisfied and
let $K_r$, $0\le r\le q-1$ or $n-k-q+1\le r\le n-k$, be the
operators from lemma \ref{21.03.07}. Then:

(i) For all $0\le r\le q-1$ and $n-k-q+1\le r\le n-k$, $I+K_r$ is a
Fredholm endomorphism of $\Cal C^0_{n,r}(M,E)$ with index zero and
\begin{equation}\label{14.4.07-}
\ke(I+K_r)\subseteq\Cal C^\infty_{n,r}(M,E)\,.\end{equation}

(ii) We have
\begin{equation}\label{14.4.07'}Z^0_{n,0}(M,E)\subseteq
\ke(I+K_0)\subseteq\Cal C^\infty_{n,0}(M,E)\end{equation} and
\begin{equation}\label{14.4.07''}
\dim Z^0_{n,0}(M,E)=\dim Z^\infty_{n,0}(M,E)<\infty\,.
\end{equation}

(iii) If $q\ge 2$ and $1\le r\le q-1$ or $n-k-q+1\le r\le n-k$, then
\begin{equation}\label{11.4.07'''}
(I+K_r)\Big(\Cal Z_{n,r}^0(M,E)\Big)\subseteq\Cal
B^{\alpha\rightarrow 0}_{n,r}(M,E)
\end{equation}
and $(I+K_r)\big\vert_{\Cal Z_{n,r}^0(M,E)}$ is a Fredholm
endomorphism with index zero of $\Cal Z_{n,r}^0(M,E)$.

(iv) If $q\ge 2$ and  $1\le r\le q-1$ or $n-k-q+1\le r\le n-k$, then
$\Cal B^{\alpha\rightarrow 0}_{n,r}(M,E)$ is a closed subspace of
finite codimension in $Z^0_{n,r}(M,E)$.

(v) If $q\ge 2$ and  $1\le r\le q-1$ or $n-k-q+1\le r\le n-k$, then
\begin{equation}\label{11.4.07''}
(I+K_r)\Big(\Cal B_{n,r}^{\alpha\rightarrow 0}(M,E)\Big)\subseteq
\Cal B_{n,r}^{\alpha\rightarrow 0}(M,E)
\end{equation}
and $(I+K_r)\big\vert_{\Cal B_{n,r}^{\alpha\rightarrow 0}(M,E)}$ is
a Fredholm endomorphism with index zero of $\Cal
B_{n,r}^{\alpha\rightarrow 0}(M,E)$.
\end{lem}

We come now to the induction step, we restrict ourselves here to the
case of large degrees (the case of small degrees is contained in
\cite{LaLeglobal}). To simplify the notations, we set $r_0=n-k-q+1$.

\begin{lem}\label{15.4.07-} Suppose, for some $0<\alpha<1$ and some integer
$q$ with $1\le q\le n-k$,  condition $H(\alpha,q)$ is satisfied and
let $T_r$, $r_0\le r\le n-k$, and $K_r$, $r_0\le r\le n-k$, be the
operators from lemma \ref{21.03.07}. Then there exist finite
dimensional continuous linear operators
\begin{equation*}\begin{split}
K_r'&:\Cal C^0_{n,r}(M,E)\rightarrow \Cal
C^\infty_{n,r}(M,E)\,,\qquad r_0\le r\le n-k,\\
K_r''&:\Cal C^0_{n,r}(M,E)\rightarrow \Cal
C^\infty_{n,r}(M,E)\,,\qquad r_0\le r\le n-k,\\ T'_{r}&:\Cal
C^0_{n,r}(M,E) \rightarrow \Cal C^\infty_{n,r-1}(M,E)\,,\qquad
r_0\le r\le n-k,\end{split}\end{equation*}such that with the
abbreviations
$$
N_r:=I+K_r+K_r'+K_r''\,,\quad r_0\le r\le n-k,
$$
each $N_r$, $r_0\le r\le n-k$, is a Fredholm endomorphism with index
zero of $\Cal C^0_{n,r}(M,E)$ (this is clear, because $I+K_r$ has
this property), and:

\vspace{3mm}(i) If $r_0\le r\le n-k$ and $f\in (\Dom
d)^0_{n,r}(M,E)$, then
\begin{equation}\label{15.4.07a}
N_rf= d(T_r+T_r')f+(T_{r+1}+T_{r+1}')df\,,
\end{equation}
and hence
\begin{equation}\label{15.4.07++} dN_{r}f=N_{r+1}df\,.
\end{equation}

\vspace{2mm}(ii) We have
\begin{equation}\label{17.4.07***} \Cal C^0_{n,r}(M,E)=
\im N_r\oplus \ke N_r\,,\qquad\text{if }\; r_0\le r\le
n-k\,,\end{equation}
\begin{equation}\label{15.4.07--}
\ke N_r\subseteq \Cal Z^\infty_{n,r}(M,E)\,,\qquad\text{if }\;
r_0\le r\le n-k\,,
\end{equation}
\begin{equation}\label{15.4.07---} \Cal Z^0_{n,r}(M,E)=
\Cal B^{\alpha\rightarrow 0}_{n,r}(M,E)\oplus \ke N_r\qquad\text{if
}\;r_0\le r\le n-k\,,
\end{equation}
and
\begin{equation}\label{20.4.07'}
\Cal B^{\alpha\rightarrow 0}_{n,r}(M,E)=\Cal B^{0\rightarrow
0}_{n,r}(M,E)\,,\qquad\text{if }\; r_0\le r\le n-k\,.\end{equation}

(iii) If $r_0\le r\le n-k$, then
\begin{equation}\label{15.4.07+++}
N_r\big(\Cal B_{n,r}^{\alpha\rightarrow 0}(M,E)\big)= \Cal
B_{n,r}^{\alpha\rightarrow 0}(M,E)\,,
\end{equation}
and $N_r\big\vert_{\Cal B^{\alpha\rightarrow 0}_{n,r}(M,E)}$ is an
isomorphism of $\Cal B^{\alpha\rightarrow 0}_{n,r}(M,E)$.

(iv) If $r_0\le r\le n-k-1$, then
\begin{equation}\label{20.4.07f}
(\Dom d)^0_{n,r}(M,E)=N_{r}\Big((\Dom d)^0_{n,r}(M,E)\Big)\oplus \ke
K_r\,,\end{equation} and hence $N_r\big\vert_{\im N_r\cap (\Dom
d)^0_{n,r}(M,E)}$ is an isomorphism of $\im N_r\cap (\Dom
d)^0_{n,r}(M,E)$.

(v) \underline{Remark:} It follows from \eqref{17.4.07***},
\eqref{15.4.07---} and \eqref{15.4.07+++} that
\begin{equation}\label{15.5.07}
\im N_r\cap\Cal Z^0_{n,r}(M,E)=\Cal B_{n,r}^{\alpha\rightarrow
0}(M,E)\qquad\text{if }\;r_0\le r\le n-k\,.
\end{equation}
\end{lem}

\begin{proof}
We proceed by induction on $r$. We first construct the operators
$K'_{r_0}$, $K''_{r_0}$, $T'_{r_0}$ and $T'_{r_0+1}$.

We begin with the construction of $K''_{r_0}$ and $T'_{r_0}$. We are
looking for an operator $K''_{r_0}$, which satisfies
$$K''_{r_0}=dT'_{r_0},$$
where $T'_{r_0}$ is a finite dimensional continuous linear operator
from $\Cal C^0_{n,r_0}(M,E)$ to $\Cal C^\infty_{n,r_0-1}(M,E)$ and,
if we set
$$\widetilde N_{r_0}:=I+K_{r_0}+K''_{r_0},$$
\begin{equation}\label{17.4.08a}
\ke \widetilde N_{r_0}\cap\Cal B^{0\rightarrow
0}_{n,r_0}(M,E)=\{0\}\,.
\end{equation}

By Lemma \ref{6.4.07a} the operator $(I+K_{r_0})\big\vert_{\Cal
B_{n,r_0}^{\alpha\rightarrow 0}(M,E)}$ is a Fredholm endomorphism
with index zero of $\Cal B_{n,r_0}^{\alpha\rightarrow 0}(M,E)$ and
hence its kernel and cokernel are finite dimensional and of the same
dimension.

Let $m=\dim \ke (I+K_{r_0})\big\vert_{\Cal B^{0\rightarrow
0}_{n,r_0}(M,E)}$. If $m=0$ set $K''_{r_0}=T'_{r_0}=0$. If $m>0$,
since $\ke(I+K_{r_0})\subseteq\Cal C^\infty_{n,r_0}(M,E)$, we can
choose a basis of $\ke (I+K_{r_0})\big\vert_{\Cal B^{0\rightarrow
0}_{n,r_0}(M,E)}$ made of $\ci$-smooth forms
$\theta_1,\dots,\theta_m$.

Moreover there exists a vector space $S$ of dimension $m$ such that
\begin{equation*}
(I+K_{r_0})(\Cal Z^0_{n,r_0}(M,E))\oplus S=\Cal B^{0\rightarrow
0}_{n,r_0}(M,E).
\end{equation*}
Let $d\lambda_1,\dots,d\lambda_m$ be a basis of $S$. Then by
Friedrichs lemma for $\cc^k$-topology (cf. Appendix) there exists
$\wt\lambda_1,\dots,\wt\lambda_m$ such that $\wt S={\rm
Vect}(d\wt\lambda_1,\dots,d\wt\lambda_m)$ satisfies also
\begin{equation}\label{inter}
(I+K_{r_0})(\Cal Z^0_{n,r_0}(M,E))\oplus \wt S=\Cal B^{0\rightarrow
0}_{n,r_0}(M,E).
\end{equation}
Taking a dual basis, we can find forms $\psi_1,\dots,\psi_m$ of
degree $n-k-r_0$ such that
\begin{equation*}
\int_M\psi_\alpha\wedge\theta_\beta=\delta_{\alpha\beta},\qquad
1\le\alpha,\beta\le m\,.
\end{equation*}
We set for $f\in\cc^0_{n,r_0}(M,E)$
\begin{equation*}
T'_{r_0}f=\sum_{\alpha=1}^m(\int_M
f\wedge\psi_\alpha)\wt\lambda_\alpha.
\end{equation*}
It follows from this definition that $T'_{r_0}$ is a finite
dimensional continuous linear operator from $\Cal C^0_{n,r_0}(M,E)$
into $\Cal C^\infty_{n,r_0-1}(M,E)$. Then we define $K''_{r_0}$ by
$K''_{r_0}=dT'_{r_0}$ and $\widetilde N_{r_0}$ by $\widetilde
N_{r_0}:=I+K_{r_0}+K''_{r_0}$. The operator $\widetilde N_{r_0}$ is
then a Fredholm operator with index $0$ of $\cc^0_{n,r_0}(M,E)$ and
by Lemma \ref{10.4.07a}
\begin{equation*}
\ke \widetilde N_{r_0}\subseteq \Cal C^\infty_{n,r_0}(M,E)\,.
\end{equation*}
It remains to prove (\ref{17.4.08a}). Let $f\in\ke \widetilde
N_{r_0}\cap\Cal B^{0\rightarrow 0}_{n,r_0}(M,E)$, then
$$(I+K_{r_0})f+K''_{r_0}f=0.$$
Since $(I+K_{r_0})f\in(I+K_{r_0})(\Cal Z^0_{n,r_0}(M,E))$ and
$K''_{r_0}f\in\wt S$, by (\ref{inter}) we get
$$f\in\ke (I+K_{r_0}\big\vert_{\Cal B^{0\rightarrow
0}_{n,r_0}(M,E)})\cap\ke K''_{r_0}$$ and hence $f=0$ by definition
of $T'_{r_0}$ and $K''_{r_0}$, which conclude the proof of
\eqref{17.4.08a}.

\vspace{2mm} Next we prove that
\begin{equation}\label{17.4.08c}
\Cal Z^0_{n,r_0}(M,E)=\Cal B^{\alpha\rightarrow
0}_{n,r_0}(M,E)\oplus\Big(\ke\widetilde N_{r_0}\cap\Cal
Z^0_{n,r_0}(M,E)\Big)\,.
\end{equation}
Since $K_{r_0}''=dT_{r_0}'$ and $\im T'_{r_0}\subseteq \Cal
C_{q-1}^{\infty}(M,E)$, it is clear that
\begin{equation}\label{17.4.07d}
\im K_{r_0}''\subseteq \Cal B_{n,r_0}^{\alpha\rightarrow 0}(M,E)\,.
\end{equation}By lemma \ref{6.4.07a} (iv),
$(I+K_{r_0})\big\vert_{\Cal B_{n,r_0}^{\alpha\rightarrow 0}(M,E)}$
is a Fredholm endomorphism with index zero of $\Cal
B_{n,r_0}^{\alpha\rightarrow 0}(M,E)$. Since $K_{r_0}''$ is finite
dimensional and we have \eqref{17.4.07d}, this implies that
$\widetilde N_{r_0}\big\vert_{\Cal B_{n,r_0}^{\alpha\rightarrow
0}(M,E)}$ has the same property. By \eqref{17.4.08a}, this means
that
\begin{equation}\label{18.4.07late}
 \widetilde N_{r_0}\big\vert_{\Cal B_{n,r_0}^{\alpha\rightarrow
0}(M,E)}\text{ is an isomorphism of }\Cal
B_{n,r_0}^{\alpha\rightarrow 0}(M,E).\end{equation} In particular,
\begin{equation}\label{17.4.07e}
\im\widetilde N_{r_0}\big\vert_{\Cal B_{n,r_0}^{\alpha\rightarrow
0}(M,E)}=\Cal B_{n,r_0}^{\alpha\rightarrow 0}(M,E)\,.
\end{equation}
Moreover, by part (iii) of lemma \ref{6.4.07a},
$(I+K_{r_0})\big\vert_{\Cal Z_{n,r_0}^{0}(M,E)}$ is a Fredholm
endomorphism with index zero of $\Cal Z_{n,r_0}^{0}(M,E)$, where
$$
\im(I+K_{r_0})\big\vert_{\Cal Z_{n,r_0}^{0}(M,E)}\subseteq \Cal
B_{n,r_0}^{\alpha\rightarrow 0}(M,E)
$$
Once again since $K_{r_0}''$ is finite dimensional and we have
\eqref{17.4.07d}, this implies that also $\widetilde
N_{r_0}\big\vert_{\Cal Z_{n,r_0}^{0}(M,E)}$ is a Fredholm
endomorphism with index zero of $\Cal Z_{n,r_0}^{0}(M,E)$, where
$$
\im \widetilde N_{r_0}\big\vert_{\Cal Z^0_{n,r_0}(M,E)}\subseteq\Cal
B_{r_0}^{\alpha\rightarrow 0}(M,E)\,.
$$ Together with \eqref{17.4.07e} this gives
\begin{equation}\label{17.4.07f}
\im \widetilde N_{r_0}\big\vert_{\Cal Z^0_{n,r_0}(M,E)}=\Cal
B_{n,r_0}^{\alpha\rightarrow 0}(M,E)\,.
\end{equation}
Therefore, \eqref{17.4.08a} can be written
\begin{equation*}
\ke \widetilde N_{r_0}\cap\im \widetilde N_{r_0}\big\vert_{\Cal
Z^0_{n,r_0}(M,E)}=\{0\}\,.
\end{equation*}
Hence
\begin{equation}\label{17.4.07g}
\im \widetilde N_{r_0}\big\vert_{\Cal Z^0_{n,r_0}(M,E)}\cap\ke
\widetilde N_{r_0}\big\vert_{\Cal Z^0_{n,r_0}(M,E)}=\{0\}\,.
\end{equation} As the index of $\widetilde N_{r_0}\big\vert_{\Cal Z^0_{n,r_0}(M,E)}$
is zero, this yields
\begin{equation*}\begin{split}
\Cal Z^0_{n,r_0}(M,E)&=\im\widetilde N_{r_0}\big\vert_{\Cal
Z^0_{n,r_0}(M,E)}\oplus\ke\widetilde N_{r_0}\big\vert_{\Cal
Z^0_{n,r_0}(M,E)}\\
&=\im\widetilde N_{r_0}\big\vert_{\Cal
Z^0_{n,r_0}(M,E)}\oplus\Big(\ke\widetilde N_{r_0}\cap\Cal
Z^0_{n,r_0}(M,E)\Big)\,.
\end{split}\end{equation*}
Again by \eqref{17.4.07f}, this proves \eqref{17.4.08c}.

\vspace{2mm} From \eqref{17.4.08a} and \eqref{17.4.08c} it follows
that
\begin{equation}\label{20.4.07''}
\Cal B^{\alpha\rightarrow 0}_{n,r_0}(M,E)=\Cal B^{0\rightarrow
0}_{n,r_0}(M,E)\,.
\end{equation}

The construction of the operators $K'_{r_0}$ and $T'_{r_0+1}$ and
the proof of there properties are exactly the same as for the
operators $K_{q-1}'$ and $T_q'$ in \cite{LaLeglobal}. We do not
repeat it here and this ends the initialization of the induction.

\vspace{2mm} Now we assume that the operators $K'_r$, $K''_r$,
$T'_r$ and $T'_{r+1}$ of Lemma \ref{15.4.07-} are construct for some
$r_0\le r\le n-k$ and that they satisfy the properties (i) to (v) of
Lemma \ref{15.4.07-}.

We set
\begin{equation}\label{17.4.07p}
K''_{r+1}:=dT_{r+1}'\,.
\end{equation}
Since $T_{r+1}'$ is a finite dimensional continuous linear operator
from $\Cal C^0_{n,r+1}(M,E)$ to $\Cal C^\infty_{n,r}(M,E)$, then it
is clear that also $K''_{r+1}$ is such an operator.  Set
$$\widetilde N_{r+1}:=I+K_{r+1}+K''_{r+1}.$$ Then, by lemma
\ref{10.4.07a}, $\widetilde N_{r+1}$ is a Fredholm endomorphism with
index zero of $\Cal C^0_{n,r+1}(M,E)$, and
\begin{equation}\label{17.4.07}
\ke \widetilde N_{r+1}\subseteq \Cal C^\infty_{n,r+1}(M,E)\,.
\end{equation}

\vspace{2mm} Now we first prove that
\begin{equation}\label{17.4.07a}
\ke \widetilde N_{r+1}\cap\Cal B^{0\rightarrow
0}_{n,r+1}(M,E)=\{0\}\,.
\end{equation}
Let $g\in \Cal B^{0\rightarrow 0}_{n,r+1}(M,E)$ with $\widetilde
N_{r+1}g=0$ be given. Take $f\in \Cal C^0_{n,r}(M,E)$ with $g=df$.
Then, by definition of $\widetilde N_{r+1}$ and $K_{r+1}''$, we get
$$
0=\widetilde N_{r+1}g=\widetilde
N_{r+1}df=(I+K_{r+1})df+K_{r+1}''df=(I+K_{r+1})df+dT_{r+1}'df\,.
$$
By \eqref{21.3.07'}, this implies
\begin{equation}\label{17.4.07b}
0=(dT_{r+1}+dT_{r+1}')df\,. \end{equation} Since, by hypothesis of
induction, the operators $K'_r$, $K''_r$, $T'_r$ and $T'_{r+1}$
satisfy statement (i) of lemma \ref{15.4.07-}, we have
\begin{equation*}
N_{r}f= d(T_{r}+T_{r}')f+(T_{r+1}+T_{r+1}')df\,.
\end{equation*}
Hence $dN_{r}f=d(T_{r+1}+T_{r+1}')df$, which implies by
\eqref{17.4.07b} that
\begin{equation}\label{16.5.07} N_{r}f\in \Cal
Z^0_{n,r}(M,E).
\end{equation}
By hypothesis of induction, \eqref{15.5.07} is valid for $r$, and
then \eqref{16.5.07}  implies that
$$
N_{r}f\in \Cal B^{\alpha\rightarrow 0}_{n,r}(M,E)\,.
$$
Therefore as, by hypothesis of induction, \eqref{15.4.07+++} is
valid for $r$, we can find $\widetilde f\in \Cal
B^{\alpha\rightarrow 0}_{n,r}(M,E)$ with
$$
N_{r}f=N_{r}\widetilde f\,.
$$
Hence $f-\widetilde f\in \ke N_{r}$ and since, by hypothesis of
induction, \eqref{15.4.07--} is valid for $r$, this implies that
$f-\widetilde f\in \Cal Z^\infty_{n,r}(M,E)$. As $\widetilde
f\in\Cal Z^0_{n,r}(M,E)$, this further implies that $ f\in \Cal
Z^0_{n,r}(M,E)$. Hence $g=df=0$. This completes the proof of
\eqref{17.4.07a}.

Next in the same way as for $r=r_0$ we get
\begin{equation}\label{17.4.07c}
\Cal Z^0_{n,r+1}(M,E)=\Cal B^{\alpha\rightarrow
0}_{n,r+1}(M,E)\oplus\Big(\ke\widetilde N_{r+1}\cap\Cal
Z^0_{n,r+1}(M,E)\Big)\,,
\end{equation}
and then the construction of the operators $K'_{r}$ and $T'_{r+1}$
is an exact repetition of the construction of the operators
$K'_{r_0}$ and $T'_{r_0+1}$.
\end{proof}

\begin{proof}[End of the proof of Theorem \ref{20.4.07a}]
We set $\Cal H_r=\ke N_r$ for all $0\le r\le q-1$ and $n-k-q+1\le
r\le n-k$. Since the operators $N_r$ are Fredholm operators and $\ke
N_r\subseteq \Cal Z^\infty_{n,r}(M,E)$, the spaces $\Cal H_r$ are
finite subspaces of $\Cal Z^\infty_{n,r}(M,E)$.

By \eqref{17.4.07***}, we have
\begin{equation}\label{19.4.07a}
\Cal C^0_{n,r}(M,E)= \im N_r\oplus\Cal H_r
\end{equation}
 and we define $P_r$ as the linear
projection in $\Cal C^0_{n,r}(M,E)$ with
\begin{equation}\label{19.4.07b}
\im P_r=\Cal H_r\quad\text{and}\quad\ke P_r=\im N_r\,,\quad 0\le
r\le q-1\,~{\rm or}~n-k-q+1\le r\le n-k\,.
\end{equation}
Since the spaces $\im N_r$ and $H_r$ are closed in the $\Cal C^0$-topology,
these projections are continuous with respect to the $\Cal
C^0$-topology. Since, by \eqref{20.4.07'} and \eqref{15.4.07+++},
$\Cal B_{n,r}^{0\rightarrow 0}(M,E)\subseteq \im N_r$, this implies
\eqref{20.4.07}.

Set
$$
\widehat N_r=N_r+P_r\,.
$$
Then, by \eqref{19.4.07a} and \eqref{19.4.07b},  $\widehat N_r$ is
an isomorphism of $\Cal C^0_{n,r}(M,E)$. If $0\le r\le q-2$ or
$n-k-q+1\le r\le n-k$, then moreover
$$
\widehat N_{r}\Big((\Dom d)^0_{n,r}(M,E)\Big)=(\Dom d)^0_{n,r}(M,E)
$$ and therefore $\widehat N_{r}\big\vert_{(\Dom d)^0_{n,r-1}(M,E)}$
is an isomorphism of $(\Dom d)^0_{n,r}(M,E)$. Indeed, since $\ke
K_r\subseteq \Cal Z^\infty_{n,r}(M,E)\subseteq (\Dom
d)^0_{n,r}(M,E)$, this follows from part (iv) of lemma
\ref{15.4.07-}.

Setting
\begin{equation}
A_r=\begin{cases} &\widehat N_{r-1}^{-1}(T_r+T_r')\,,\qquad 1\le
r\le q \\
&(T_r+T_r')\widehat N_{r}^{-1}\,,\qquad n-k-q+1\le r\le n-k\,,
\end{cases}
\end{equation}
now we define the continuous linear operators
\begin{equation*}
A_r:\Cal C^0_{n,r}(M,E)\longrightarrow \Cal
C^0_{n,r-1}(M,E)\,,\qquad 1\le r\le q~{\rm or}~n-k-q+1\le r\le
n-k\,.
\end{equation*}

\vspace{3mm} {\em Proof of (i):} For $0\le r\le q-1$ the proof of
the assertions (i) and (ii) of the theorem is contained in section 6
of \cite{LaLeglobal}.

Let $n-k-q+1\le r\le n-k$ and $l\in \N$ be given. By definition,
$\widehat N_{r}$ is of the form $\widehat N_{r}=I+R$ where
$R\big\vert_{\Cal C^l_{n,r}(M,E)}$ is a continuous linear operator
from $\Cal C^l_{n,r}(M,E)$ to $\Cal C^{l+\alpha}_{n,r}(M,E)$. Since
$(T_r+T_r')\big\vert{\Cal C^l_{n,r}(M,E)}$ is a continuous from
$\Cal C^l_{n,r}(M,E)$ to $\Cal C^{l+\alpha}_{n,r-1}(M,E)$, it
follows that
$$
A_r=\widehat (T_r+T_r')N_r^{-1}
$$ is continuous
from $\Cal C^l_{n,r}(M,E)$ to $\Cal C^{l+\alpha}_{n,r-1}(M,E)$.

\vspace{2mm} {\em Proof of (ii):} Let $n-k-q+1\le r\le n-k$. We
first prove that
\begin{equation}\label{20.4.07i}
\widehat N_{r+1}^{-1}d\big\vert_{(\Dom d)^0_{n,r}(M,E)}=d\widehat
N^{-1}_{r}\big\vert_{(\Dom d)^0_{n,r}(M,E)}\,.
\end{equation}
Since  $\widehat N_{r}\big\vert_{(\Dom d)^0_{n,r}(M,E)}$ is an
isomorphism of $(\Dom d)^0_{n,r}(M,E)$, this is equivalent to
\begin{equation}\label{20.4.07ii}
d\widehat N_{r-1}\big\vert_{(\Dom d)^0_{n,r-1}(M,E)}=\widehat
N_rd\big\vert_{(\Dom d)^0_{n,r-1}(M,E)}\,.
\end{equation}
Let $g\in (\Dom d)^0_{n,r}(M,E)$ be given. Then, by \eqref{20.4.07},
\eqref{15.4.07---} and \eqref{15.4.07++},
$$
dP_{r-1}g=0\,,\quad P_rdg=0\quad\text{and}\quad dN_{r-1}g=N_rdg\,.
$$
Hence
$$
d\widehat
N_{r-1}g=dN_{r-1}g+dP_{r-1}g=dN_{r-1}g=N_rdg=N_rdg+P_rdg=\widehat
N_rdg\,.
$$

Now consider $f\in (\Dom d)^0_{n,r}(M,E)$. Then, by \eqref{15.4.07a}
and definition of the operators $A_r$,
\begin{equation}\label{19.4.07p}
\widehat N_{r}N_{r}^{-1}f= dA_rf+(T_{r+1}f+T'_{r+1})dN_{r}^{-1}f
\end{equation}
Since $\im N_{r}=\ke P_r$, we have $P_rN_r=0$. Hence
$(I-P_r)\widehat N_r=(I-P_r)(N_r+P_r)=N_r$ and therefore
$N_r\widehat N_r^{-1}=I-P_r$. Therefore, \eqref{19.4.07p} takes the
form
\begin{equation}\label{19.4.07q}
f-P_rf=dA_rf+(T_{r+1}f+T'_{r+1})dN_{r}^{-1}f
\end{equation}
and together with \eqref{20.4.07i} this gives \eqref{21.3.07'''}.
\end{proof}

As a direct consequence of Theorem \ref{20.4.07a}, we obtain a new
proof of the Dolbeault isomorphism for the $\opa_b$-cohomology and
of the regularity theorem for the tangentiel Cauchy-Riemann operator
in compact CR manifolds

\begin{cor}\label{dolbeault}
Suppose, for some $0<\alpha<1$ and some integer $q$ with $1\le q\le
n-k$,  condition $H(\alpha,q)$ is satisfied. For all $1\le r\le q$
or $n-k-q+1\le r\le n-k$ and $l\in \N\cup\{\infty\}$, the space
 $ \Cal B_{n,r}^{l+\alpha\rightarrow l}(M,E)$ is closed in $\Cal
C_{n,r}^{l}(M,E)$,
\begin{equation}\label{21.4.07-} \Cal
B_{n,r}^{l+\alpha\rightarrow l}(M,E)=\Cal B_{n,r}^{0\rightarrow
l}(M,E)\,,
\end{equation}
and the natural map
\begin{equation}\label{21.4.07+}
\frac{\Cal Z^\infty_{n,r}(M,E)}{d\Cal \Cal
C^\infty_{n,r-1}(M,E)}\rightarrow\frac{\Cal Z^l_{n,r}(M,E)}{\Cal
B_{n,r}^{0\rightarrow l}(M,E)} \end{equation} is injective.

If $1\le r\le q-1$ or $n-k-q+1\le r\le n-k$, then moreover, for all
$l\in \N\cup\{\infty\}$, there exist finite dimensional subspaces
$\Cal H_r$ of $\Cal Z^\infty_{n,r}(M,E)$ such that
\begin{equation}\label{19.4.07o}
 \Cal Z^l_{n,r}(M,E)=
\Cal B_{n,r}^{l+\alpha\rightarrow l}(M,E)\oplus\Cal H_r\,,
\end{equation} and hence \eqref{21.4.07+} is an isomorphism and the
cohomology groups in \eqref{21.4.07+} are finite dimensional.
\end{cor}

\begin{proof}
Let $1\le r\le q$ or $n-k-q+1\le r\le n-k$. It follows from
\eqref{21.3.07'''} and \eqref{20.4.07} that
\begin{equation}\label{19.4.07''''}
dA_rf=f\qquad\text{for all }f\in \Cal B_{n,r}^{0\rightarrow
0}(M,E)\,,
\end{equation}
and \eqref{21.4.07-} follows from \eqref{19.4.07''''} and
\eqref{19.4.07''}.

Let $1\le r\le q-1$ or $n-k-q+1\le r\le n-k$ and $l\in
\N\cup\{\infty\}$. From Lemma \ref{21.03.07}, we get that
$(I+K_{r})\big\vert_{\Cal Z^l_{n,{r}}(M,E)}$ is a Fredholm operator
with index zero of $\Cal Z^l_{n,{r}}(M,E)$ with $(I+K_{r})(\Cal
Z^l_{n,{r}}(M,E))\subseteq \Cal B_{n,{r}}^{l+\alpha\rightarrow
l}(M,E)$. Since $\Cal B_{n,{r}}^{l+\alpha\rightarrow l}(M,E)$ is the
image of a closed linear operator this implies that $\Cal
B_{n,{r}}^{l+\alpha\rightarrow l}(M,E)$ is a closed subspace for the
$\cc^l$-topology.

To prove that $\Cal B^{l+\alpha\rightarrow l}_{n,q}(M,E)$ is closed
in the $\Cal C^l$-topology, we consider a sequence $f_\nu\in \Cal
B^{l+\alpha\rightarrow l}_{n,q}(M,E)$ which converges in the $\Cal
C^l$-topology to some $f\in \Cal C^{l}_{n,q}(M,E)$. Since, by part
(i) of Theorem \ref{20.4.07a}, $A_r$ is continuous as operator from
$\Cal C^{l}_{n,q}(M,E)$ to $\Cal C^{l+\alpha}_{n,q-1}(M,E)$, then
the sequence $A_qf_\nu$ converges in the $\Cal
C^{l+\alpha}$-topology to some $g\in \Cal
C^{l+\alpha}_{n,q-1}(M,E)$, where, by \eqref{19.4.07''''},
$dA_qf_\nu=f_\nu$ for all $\nu$. Since the operator
$$
d:\Cal C^{l+\alpha}_{n,q-1}(M,E)\longrightarrow \Cal
B^{l+\alpha\rightarrow l}_{n,q}(M,E)
$$ is closed, this implies that $dg=f$, i.e. $f\in \Cal
B^{l+\alpha\rightarrow l}_{n,q}(M,E)$.

Since, by \eqref{19.4.07''},
$$A_r\Big(\Cal C^\infty_{n,r}(M,E)\cap\Cal B_{n,r}^{0\rightarrow
0}(M,E) \Big)\subseteq \Cal C^\infty_{n,r-1}(M,E)\,,$$ it follows
from \eqref{19.4.07''''} that
$$d\Cal
C^\infty_{n,r-1}(M,E)=\Cal C^\infty_{n,r}(M,E)\cap  \Cal
B^{0\rightarrow l}_{n,r}(M,E)\,,$$ which means that \eqref{21.4.07+}
is injective.

Now let $1\le r\le q-1$ or $n-k-q+1\le r\le n-k$, we define $\Cal
H_r$ by $\im P_r=\Cal H_r$, where $P_r$ is the projector from
Theorem \ref{20.4.07a}. Then, by \eqref{15.4.07---},
\eqref{20.4.07'},
$$
Z^0_{n,r}(M,E)=\Cal B^{0\rightarrow 0}_{n,r}(M,E)\oplus \Cal H_r\,.
$$ Since $\Cal H_r\subseteq \Cal
C^{\infty}_{n,r}(M,E)$, this implies that
$$
Z^l_{n,r}(M,E)=\Cal B^{0\rightarrow l}_{n,r}(M,E)\oplus \Cal
H_r\qquad\text{for all }l\in\N\cup\{\infty\}\,.
$$
By \eqref{21.4.07-} this means \eqref{19.4.07o}.
\end{proof}

\section{Appendix}\label{ap}

In the 40's, Friedrichs has proven a density lemma for the
$L^2$-topology for partial differential operators in $\rb^n$. If $P$
is such an operator he proves that the $\ci$-smooth functions are
dense in the domain of definition of $P$ for the graph norm. This
result has been later extend to the $L^p$-topology, $1\le p<\infty$.
Here we want to generalize it to differential operators between
vector bundles for the $\cc^k$-topology.

Such differential operators between vector bundles appears
naturally, for example the tangential Cauchy-Riemann operator on a
CR generic submanifold of a complex manifold is a differential
operator between two bundles of differential forms.

Let $X$ be a paracompact differential manifold of class $\ci$ of
real dimension $n$ and $E$ and $F$ two vector bundles of class
$\ci$, respectively of rank $p$ and $q$.

Let $\uc=(U_i)_{i\in I}$ be a locally finite open covering of $X$ by
coordinates domains which are also trivialization domains for both
$E$ and $F$ and $(M_{ij})_{i,j\in I}$ be the transition matrices of
$E$ on $(U_{ij}=U_i\cap U_j)_{i,j\in I}$ and $(N_{ij})_{i,j\in I}$
be the transition matrices of $F$ on $(U_{ij}=U_i\cap U_j)_{i,j\in
I}$.

For $k\in\N\cup\{\infty\}$ we denote by $\Gamma^{k}(X,E)$,
respectively $\Gamma^{k}(X,F)$, the vector space of $\cc^{k}$-smooth
sections of $E$, respectively $F$. Some trivialization being given
on $\uc$, for each $f\in \Gamma^k(X,E)$, $f\big\vert_{U_i}$ is given
by a $p$-vector of functions $f_i=(f_i^1,\dots,f_i^p)$ and on
$U_i\cap U_j$ we have $f_i=M_{ij}f_j$.

A linear differential operator $P$ of order $1$ and class $\cc^k$
between the fiber bundles $E$ and $F$ is a linear map between
$\Gamma^{k+1}(X,E)$ and $\Gamma^{k}(X,F)$ given in some
trivialization of $E$ and $F$ by a family of linear operators $P_i$
from $\Gamma^{k+1}(U_i,E)$ into $\Gamma^{k}(U_i,F)$, $i\in I$, such
that, for $f\in\Gamma^{k+1}(X,E)$,

(i) $P_i(f_i)=N_{ij}P_j(M_{ij}f_j)$ on $U_i\cap U_j$,

(ii) each $P_i$ is a system of partial differential linear equations
of order $1$, i.e. $P_i=(L_i^{rs})$ is given by a $(q,p)$-matrice of
partial differential linear equations, where, if $(x_1,\dots,x_n)$
denotes some coordinates on $U_i$, $D_l=\frac{\pa}{\pa x_l}$, $1\le
l\le n$, the partial derivative relatively to the coordinate $x_l$,
$a_l^{rs}$, $1\le l\le n$, $1\le r\le q$ and $1\le s\le p$,
$\cc^{k+1}$ functions on $\rb^n$ and $a_0^{rs}$, $1\le r\le q$ and
$1\le s\le p$, $\cc^{k}$ functions on $\rb^n$, then
$$L_i^{rs}=\sum_{l=1}^n a_l^{rs}D_l + a_0^{rs}\,.$$

Let $P$ be a linear differential operator $P$ of order $1$ and class
$\cc^k$ between the fiber bundles $E$ and $F$, we shall say that a
section $f$ of class $\cc^k$ of $E$ is in the domain of definition
of $P$, $\Dom P$, if $Pf$, which is defined in the sense of
distributions, belongs to $\Gamma^{k}(X,F)$.

If $K$ is a compact subset of $X$ and $f\in\Dom P$, the graph
$\cc^k$-norm on $K$ of $f$ is defined by
$$\|f\|_{gr(K,k)}=\|f\|_{K,k}+\|Pf\|_{K,k}\,.$$

\begin{thm}\label{friedrichs}
{\rm Friedrichs'lemma for the $\cc^k$-topology.} Let $P$ be a linear
differential operator of order $1$ and class $\cc^k$ between two
fiber bundles $E$ and $F$ of class $\ci$ over a differential
manifold $X$ of class $\ci$. For each compact subset $K$ of $X$, the
$\ci$-smooth sections of $E$ are dense in the domain of definition
of $P$ for the graph $\cc^k$-norm on $K$.
\end{thm}

\begin{proof}
First consider the case when $X=\rb^n$, $E$ and $F$ are both trivial
bundles of rank respectively $p$ and $q$ and $P=(L^{rs})$ is a
system of partial differential equations. In particular for $p=q=1$
we get

\begin{lem}\label{pde}
Let $L=\sum_{l=1}^n a_lD_l+ a_0$ with $a_l\in\cc^{k+1}(\rb^n)$,
$1\le l\le n$, and $a_o\in\cc^k(\rb^n)$ and
$\varphi\in\cc^\infty(\rb^n)$ be a positive smooth function with
compact support in the unit ball of $\rb^n$ such that
$\int_{\rb^n}\varphi~dx=1$. Set for $\epsilon>0$ and $x\in\rb^n$,
$$\varphi_\epsilon (x)=\frac{1}{\epsilon}\varphi(\frac{x}{\epsilon^n})\,.$$
Then for each compact subset $K$ of $\rb^n$ and any $f\in\Dom(L)$
$$\|L(f*\varphi_\epsilon)-L(f)\|_{K,k}\to 0\,,$$
when $\epsilon\to 0$.
\end{lem}
\begin{proof}
Let $f\in\Dom(L)$ and set $L(f)=\wt L(f)+a_0 f$. By the classical
convergence properties of the convolution, we know that
$\|a_0(f*\varphi_\epsilon)-a_0f\|_{K,k}$ tends to $0$ when
$\epsilon$ tends to $0$. By the triangle inequality
$$\|\wt L(f*\varphi_\epsilon)-\wt L(f)\|_{K,k}\leq
\|\wt L(f*\varphi_\epsilon)-\wt L(f)*\varphi_\epsilon\|_{K,k}+\|\wt
L(f)*\varphi_\epsilon-\wt L(f)\|_{K,k}\,.$$ It follows again from
the classical properties of the convolution that, since $\wt
L(f)\in\cc^k(\rb^n)$, $\|\wt L(f)*\varphi_\epsilon-\wt L(f)\|_{K,k}$
tends to $0$, when $\epsilon$ tends to $0$. Now by definition of
$\wt L$, we have
\begin{equation}\label{reduction}
\|\wt L(f*\varphi_\epsilon)-\wt L(f)*\varphi_\epsilon\|_{K,k}\leq
\sum_{l=1}^n
\|a_lD_l(f*\varphi_\epsilon)-(a_lD_l(f))*\varphi_\epsilon\|_{K,k}\,.
\end{equation}

Therefore it is sufficient to prove that each term in
\eqref{reduction} converges to $0$ when $\epsilon$ tends to $0$. Let
$s=(s_1,\dots,s_n)\in\N^n$ be a multi-index of length less or equal
to $k$, we set $D^s=D_1^{s_1}\dots D_n^{s_n}$. Then
$$\|f\|_{K,k}=\sum_{|s|\le k}\|D^s\|_{K,0}\,.$$
From the Leibniz formula, we deduce
$$D^s(a_lD_l(f))=\sum_{r\le s} \binom{s}{r}(D^r a_l)(D_l(D^{s-r}f)$$
and the proof of the lemma is then reduced to  the following fact~:

($\clubsuit$)\hspace{5mm} Let $a\in\cc^1(\rb^n)$, $1\le l\le n$ and
$f\in\cc(\rb^n)$, then the smooth function
$aD_l(f*\varphi_\epsilon)-(aD_l(f))*\varphi_\epsilon$ converges
uniformly to  $0$ on $K$, when $\epsilon$ tends to $0$.

Note that it is clearly the case if the function $f$ is of class
$\ci$, because then
$D_l(f*\varphi_\epsilon=D_l(f)*\varphi_\epsilon)$. Let
$g\in\ci(\rb^n)$, then by linearity
\begin{equation}\label{dense}
aD_l(f*\varphi_\epsilon)-(aD_l(f))*\varphi_\epsilon=aD_l(g*\varphi_\epsilon)-(aD_l(g))*\varphi_\epsilon
+aD_l((f-g)*\varphi_\epsilon)-(aD_l(f-g))*\varphi_\epsilon
\end{equation}
Assume we can prove that there exists a constant $C$ independent of
$\epsilon\in ]0,1]$ such that
\begin{equation}\label{fond}
\|aD_l(f*\varphi_\epsilon)-(aD_l(f))*\varphi_\epsilon\|_{K,0}\le
C\|f\|_{K+B(0,1),0}\,,
\end{equation}
then from \eqref{dense} we get
\begin{equation*}
\|aD_l(f*\varphi_\epsilon)-(aD_l(f))*\varphi_\epsilon\|_{K,0}\le
\|aD_l(g*\varphi_\epsilon)-(aD_l(g))*\varphi_\epsilon\|_{K,0}
+C\|f-g\|_{K+B(0,1),0}\,.
\end{equation*}
and, since continuous functions in $\rb^n$ can be uniformly
approximate on $K+B(0,1)$ by $\ci$-smooth functions, this concludes
the proof of ($\clubsuit$) by classical arguments.

Now let us prove \eqref{fond}. First let us recall that
$$f*\varphi_\epsilon=\frac{1}{\epsilon^n}\int_X
f(x-y)\varphi(\frac{y}{\epsilon})~dy=\int_X f(x-\epsilon
y)\varphi(y)~dy\,.$$
Then
\begin{equation*}
aD_l(f*\varphi_\epsilon)-(aD_l(f))*\varphi_\epsilon=\int_X
\big(a(x)-a(x-\epsilon y)\big)(D_l f)(x-\epsilon y)\varphi(y)~dy\,,
\end{equation*}
and after an integration by parts
\begin{align*}
aD_l(f*\varphi_\epsilon)-(aD_l(f))*\varphi_\epsilon&=\int_X
\frac{1}{\epsilon}\big(a(x)-a(x-\epsilon y)\big) f(x-\epsilon
y)\frac{\pa\varphi}{\pa y_l}(y)~dy\\
&+\int_X \frac{\pa a}{\pa y_l}(x-\epsilon y)f(x-\epsilon
y)\varphi(y)~dy\,,
\end{align*}
since $\frac{\pa}{\pa y_l}(f(x-\epsilon y))=-\epsilon (D_l
f)(x-\epsilon y)$.

As $a$ is of class $\cc^1$, for $\epsilon\le 1$, there exists a
constant $M_K$ such that for $x\in K$ and $y\in \supp \varphi\subset
B(0,1)$ we have
$$|a(x)-a(x-\epsilon y)|\leq M_K \epsilon |y|\qquad {\rm and}\qquad
|\frac{\pa a}{\pa y_l}(x-\epsilon y)|\le M_K\,,$$ and therefore
\begin{equation*}
\|aD_l(f*\varphi_\epsilon)-(aD_l(f))*\varphi_\epsilon\|_{K,0}\le M_K
\big(\int_X (|y||\frac{\pa\varphi}{\pa
y_l}(y)|+|\varphi(y)|)~dy\big)\|f\|_{K+B(0,1),0}\,.
\end{equation*}
Setting $C=M_K \big(\int_X (|y||\frac{\pa\varphi}{\pa
y_l}(y)|+|\varphi(y)|)~dy\big)$, we get \eqref{fond}.
\end{proof}

Now if $P$ is given by a $(q,p)$-matrice $(L^{rs})$, using the
triangle inequality, we deduce easily from Lemma \ref{pde} that for
each compact subset $K$ of $\rb^n$ and any $f\in\Dom(P)$
$$\|P(f*\varphi_\epsilon)-P(f)\|_{K,k}\to 0\,,$$
when $\epsilon\to 0$, which proves Theorem \ref{friedrichs} for a
system.

\vspace{2mm} Now we have to globalize the situation.

Let $(\chi_i)_{i\in I}$ be a partition of the unity subordinated to
the open covering $\uc$ and $f\in\Gamma^k(X,E)$. After a choice of
coordinates in each $U_i$ and of a trivialization of $E$, we can
define $\chi_i f$ as a $p$-vector $f_i=(f_i^1,\dots,f_i^p)$ of
functions with compact support in $\rb^n$. Then, for $\epsilon$
sufficiently small,
$f_i*\varphi_\epsilon=(f_i^1*\varphi_\epsilon,\dots,f_i^p*\varphi_\epsilon)$
can be identify with a section $(f_i)_\epsilon$ of $E$ with compact
support in $U_i$. Set $f_\epsilon=\sum_{i\in I}(f_i)_\epsilon$, then
$f_\epsilon\in\Gamma^\infty(X,E)$ and for each compact subset of $X$
$\|f-f_\epsilon\|_{K,k}$ tends to $0$ when $\epsilon$ tends to $0$.

Moreover on one hand $Pf=\sum_{i\in I}P(\chi_i f)$ and, after a
choice of trivialization of $F$ over $U_i$, $P(\chi_i f)=P_i (f_i)$
 and on the other hand
$Pf_\epsilon=\sum_{i\in I}P(( f_i)_\epsilon)$ and, for $\epsilon$
sufficiently small, $P(( f_i)_\epsilon)=P_i(f_i*\varphi_\epsilon)$.
By the case of a system, which has been studied previously, $\|P_i
(f_i)-P_i(f_i*\varphi_\epsilon)\|_{K,k}$ tends to $0$ when
$\epsilon$ tends to $0$ and therefore $\|Pf-Pf_\epsilon\|_{K,k}$
tends to $0$ when $\epsilon$ tends to $0$, which proves the theorem.
\end{proof}

\providecommand{\bysame}{\leavevmode\hbox to3em{\hrulefill}\thinspace}
\providecommand{\MR}{\relax\ifhmode\unskip\space\fi MR }
\providecommand{\MRhref}[2]{%
  \href{http://www.ams.org/mathscinet-getitem?mr=#1}{#2}
}
\providecommand{\href}[2]{#2}

\vespa\vespa
\begin{flushright}
  \begin{minipage}[t]{8cm}

Department of Mathematics

Imperial College London

180 Queen's Gate

LONDON SW7 2AZ

United Kingdom

till.broennle07@imperial.ac.uk

\vespa

Universit\'e de Grenoble

   Institut Fourier

UMR 5582  CNRS/UJF

   BP 74

38402 St Martin d'H\`eres Cedex

France

Christine.Laurent@ujf-grenoble.fr

 \vespa

Institut f\"{u}r Mathematik

HUMBOLDT Universit\"{a}t zu Berlin

Rudower Chaussee 25

D-12489 Berlin 

Germany

leiterer@math.hu-berlin.de

  \end{minipage}
\end{flushright}

\enddocument

\end